\RequirePackage{fix-cm}
\documentclass[smallextended]{svjour3}       
\smartqed  
\usepackage{graphicx}
\usepackage{amsmath,amssymb,amsfonts,amscd}
\usepackage{xcolor}

\numberwithin{theorem}{section} 
\numberwithin{definition}{section}
\numberwithin{remark}{section}
\numberwithin{corollary}{section} 
\numberwithin{lemma}{section} 
\numberwithin{example}{section} 
\numberwithin{equation}{section}

\begin{document}

\title{Yeh--Fourier--Feynman transforms and convolutions
associated with Gaussian processes}


\titlerunning{Yeh--Fourier--Feynman transforms and convolutions}       

\author{Jae Gil Choi}
 

 \institute{Jae Gil Choi  (Corresponding author)\at
           School of General Education, Dankook University,  Cheonan {\color{red}31116, Republic of Korea}\\
           \email{jgchoi@dankook.ac.kr}
}

\date{Received: date / Accepted: date}

\maketitle

\begin{abstract}
In this paper we study an analytic Yeh--Feynman integral and 
an analytic Yeh--Fourier--Feynman transform   associated with
Gaussian processes. Fubini theorems involving  the generalized analytic  
Yeh--Feynman integrals are established. The Fubini theorems investigated 
in this paper are to express the iterated generalized Yeh--Feynman integrals 
associated with Gaussian processes as a single  generalized Yeh--Feynman integral.
Using our Fubini theorems, we next examined fundamental relationships (with extended versions) 
between generalized Yeh--Fourier--Feynman transforms and convolution products 
(with respect to Gaussian processes) of functionals on Yeh--Wiener space.
\end{abstract}

\section{Introduction}\label{sec:intro}

\par
Given a positive real $T>0$, let  $C_0[0,T]$ denote the one-parameter Wiener 
space, that is, the space of  all  real-valued continuous functions $x$ 
on the compact  interval $[0,T]$  with $x(0)=0$. As mentioned in \cite{HSS01-F},
the usual Fubini theorem  does not apply to analytic Wiener and Feynman 
integrals since they are not defined in terms of a countably additive nonnegative measure. 
Rather, they are defined in terms of a process of analytic continuation and a limiting
procedure applied to a Wiener integral which is based on such a measure, 
see \cite{Cameron,SS04}.
Thus, in \cite{HSS01-F,HSS01-I},  Huffman,  Skoug, and  Storvick investigated 
the structure of the Fubini theorem for analytic Feynman integrals and 
analytic Fourier--Feynman transforms of functionals on the classical Wiener space $C_0[0,T]$.
The Fubini theorems for the  analytic Feynman integral presented in \cite{HSS01-F,HSS01-I} also
are effected by the concept of the scale-invariant measurability \cite{CM47,JS79} in $C_0[0,T]$.

\par
In \cite{Kitagawa},  Kitagawa introduced a function space which is the collection of 
the continuous functions of two variables,  $x(s,t)$, on the unit square $[0,1]\times[0,1]$
satisfying $x(s,0)=x(0,t)=0$ for all $ s,t \in[0,1]$, and he investigated the integration on this space. 
In \cite{Yeh60}, Yeh developed the structure of the measure on this space and made 
a  logical foundation on this space. We call this space a Yeh--Wiener space and the integral a Yeh--Wiener integral.

\par
The Fubini theorems studied in  \cite{HSS01-F,HSS01-I} are related only to  
the variance parameter defining the analytic Feynman integral. 
The aim of this paper is  to establish  a Fubini theorem for the 
generalized  Yeh--Feynman integral   of functionals on the Yeh--Wiener space. 
The definition of the generalized  Yeh--Feynman integral
is based on  the Yeh--Wiener integral of functionals in  sample paths of Gaussian process
$\mathcal Y_h$ on the  Yeh--Wiener space $(C_0([0,S]\times[0,T]),m_{\mathrm y})$.
The generalized  Yeh--Feynman integral is defined as follows:
\[
\int_{C_0([0,S]\times[0,T])}F(\mathcal Y_h (x;\cdot,\cdot)) d m_{\mathrm {y}}(x),
\]  
where $\mathcal Y_h$  is the Gaussian process on $C_0([0,S]\times[0,T])\times [0,S]\times[0,T]$ 
given by the stochastic integral $\mathcal Y_h (x;s,t)=\int_0^t\int_0^s h(\nu,\tau)dx(\nu,\tau)$,
and where $h$ is a nonzero function in $L_2([0,S]\times[0,T])$ and
$\int_0^t\int_0^s h(\nu,\tau)dx(\nu,\tau)$ denotes the Paley--Wiener--Zygmund stochastic 
integral \cite{CAC88,PWZ33,PS88-Nagoya,PS88}.
 The concept  of the generalized Yeh--Wiener integral was introduced by Park and Skoug in  \cite{PS93}, and further
developed in \cite{PS94}.

\par
The  Gaussian processes used in this paper, 
as well as in \cite{PS93,PS94}, are generally non-stationary processes.
The basic structure of Fubini theorems investigated in this paper 
are related to the kernel functions $h$  in Gaussian processes $\mathcal Y_h$
defining the generalized Yeh--Feynman integral. 
  
\par
Let $\mathbb E^m$ be a  Euclidean space.
For $f\in L_1(\mathbb E^m)$, let the Fourier transform of $f$ be given by
\[
\mathcal{F}(f)(\vec u)=\int_{\mathbb E^m}e^{i\vec u  \cdot  \vec v}f(\vec v)dm_L^{\mathfrak{n}}(\vec v)
\]
and for $f, g\in L_1(\mathbb E^m)$, let the convolution of $f$ and $g$ be 
given by
\[
(f*g)(\vec u)=\int_{\mathbb E^m} f(\vec u-\vec v)g(\vec v)dm_L^{\mathfrak{n}}(\vec v)
\]
where $\vec u \cdot \vec v$ denotes the dot product of vectors $\vec u$ and $\vec v$ in 
$\mathbb E^m$,  and $dm_L^{\mathfrak{n}} (\vec v)$ denotes the normalized Lebesgue measure 
$(2\pi)^{-m/2}dv$ on $\mathbb E^m$. As commented in \cite{CC09}, the Fourier 
transform  $\mathcal{F}$ acts like a homomorphism with convolution $*$ 
and ordinary multiplication on $L_1(\mathbb E^m)$ as follows:
 for $f, g \in L_1(\mathbb E^m)$ 
\begin{equation}\label{worthy}
\mathcal{F}(f*g)=\mathcal{F}(f)\mathcal{F}(g).
\end{equation}
Also, the Fourier transform  $\mathcal{F}$ and the convolution $*$  have a dual 
property such as 
\begin{equation}\label{eq:F-02}
\mathcal{F}(f)*\mathcal{F}(g)=\mathcal{F}(f g).
\end{equation}
  
\par
In view of equations \eqref{worthy} and \eqref{eq:F-02}, it is worth-while 
to study a fundamental relation between  the  generalized  transforms 
and the generalized  convolutions  
for functionals on infinite dimensional Banach spaces.
In this view point,
Huffman,  Park, Skoug and Storvick   \cite{HPS95,HPS96,HPS97-1,PSS98}  established  fundamental 
relationships between the analytic  Fourier--Feynman transform  
and the corresponding convolution product  for  functionals $F$ and $G$ on the classical Wiener space 
$C_0[0,T]$,   as follows: 
\begin{equation}\label{eq:offt-ocp}
T_{q}^{(1)}\big((F*G)_q\big)(y) 
= T_{q}^{(1)}(F)\bigg(\frac{y}{\sqrt2}\bigg)
T_{q}^{(1)}(G)\bigg(\frac{y}{\sqrt2}\bigg)  
\end{equation}
and 
\begin{equation}\label{eq:ocp-offt}
\big(T_{q}^{(1)}(F)*T_{q}^{(1)}(G)\big)_{-q}(y) 
= T_{q}^{(1)}\bigg(F\bigg(\frac{\cdot}{\sqrt2}\bigg)G \bigg(\frac{\cdot}{\sqrt2}\bigg)\bigg)  
\end{equation}
for scale-invariant almost every $y\in C_{0}[0,T]$,
where $T_{q}^{(1)}(F)$ and  $(F*G)_q$ denote the analytic Fourier--Feynman transform (FFT)
and the convolution product (CP), respectively, 
of functionals $F$ and $G$ on $C_0[0,T]$.
For an elementary introduction of the analytic  FFT
and the corresponding CP, see \cite{SS04}. 
Equations \eqref{eq:offt-ocp} and \eqref{eq:ocp-offt} above 
are  natural extensions (to the cases on an infinite dimensional Banach space) 
of the equations \eqref{worthy} and \eqref{eq:F-02}, respectively.

\par
Since then, in \cite{HPS97-2}, the authors extended the relationships \eqref{eq:offt-ocp}
and \eqref{eq:ocp-offt} to the cases between the generalized FFT  
and the generalized  CP  associated with  Gaussian processes  on $C_0[0,T]$.
The definition of the ordinary FFT and the corresponding CP are 
based on the ordinary Wiener integral,
see \cite{HPS95,HPS96,HPS97-1}, and  the definition of the generalized FFT
and the corresponding CP studied in \cite{HPS97-2}
are  based on the generalized Wiener integral \cite{CPS93,PS91}.

The second aim of  this paper, as  applications of the Fubini theorem for the generalized Yeh--Feynman integrals 
on $C_0([0,S]\times[0,T])$, is   to investigate fundamental relationships, such as  \eqref{eq:offt-ocp}  and \eqref{eq:ocp-offt},
between the generalized Fourier--Yeh--Feynman transform (GFYFT) and the generalized CP (GCP)
associated with Gaussian processes on the Yeh--Wiener space $C_0([0,S]\times[0,T])$.

\section{Definitions and preliminaries}\label{sec:pre}

\par
Yeh--Wiener space \cite{Yeh60} is the two parameter Wiener space $(C_0(Q),\mathcal B(C_0(Q)),$ 
$m_{\mathrm{y}})$
where $Q$ is the compact rectangle $[0,S]\times[0,T]$ with nonzero area in $\mathbb R^2$,
$C_0(Q)$  is the space of all real-valued continuous functions $x$ on $Q$ such that 
$x(s,0)=x(0,t)=0$ for all $(s,0)$ and $(0,t)$ in $Q$, $\mathcal B(C_0(Q))$
is the Borel $\sigma$-field induced by the uniform norm on $C_0(Q)$,
and $m_{\mathrm{y}}$ denotes the Yeh--Wiener measure, see \cite{Kitagawa,PS93,PS94,Yeh60}.
The sample functions $x$ in $C_0(Q)$ are often called 
Brownian surfaces or Brownian sheets.

\par
Let $\mathcal W(C_0(Q))$ be the class of $m_{\mathrm{y}}$-Carath\'eodory measurable subsets 
of $C_0(Q)$. It is well known that $\mathcal W(C_0(Q))$ coincides with 
$\sigma(\mathcal B (C_0(Q)))$, the completion of the Borel $\sigma$-field $\mathcal B (C_0(Q))$.
A subset $E$ of $C_0(Q)$  is said to be scale-invariant measurable \cite{Chung87,JS79} provided 
$\rho E$ is  $\mathcal W(C_0(Q))$-measurable  for every  $\rho>0$,  and a scale-invariant measurable subset 
$N$ of $C_0(Q)$ is said to be  scale-invariant null  provided $m_{\mathrm{y}}(\rho N)=0$ for every $\rho>0$.  
A property that holds except on a scale-invariant null set is  said to hold scale-invariant 
almost everywhere (s-a.e.). A functional $F$ on $C_0(Q)$ is said to be scale-invariant measurable 
provided $F$ is defined on a scale-invariant measurable set and $F(\rho\,\cdot\,)$ is  
$\mathcal W(C_0(Q))$-measurable for every $\rho>0$.

\par
The Paley--Wiener--Zygmund (PWZ) stochastic integral 
\cite{CAC88,PWZ33,PS88-Nagoya,PS88} plays a key role throughout this 
paper. Let $\{\phi_n\}$ be a complete  orthonormal set in $L_2(Q)$, each of whose 
elements is  of bounded  variation  in the sense of Hardy--Krause \cite{BG} on $Q$.
Then for each $v\in L_2(Q)$, the PWZ stochastic 
integral $\langle{v,x}\rangle$  is defined by 
the formula
\[
\langle{v,x}\rangle
=\lim\limits_{n\to \infty} 
\int_Q\sum\limits_{j=1}^n (v,\phi_j)_2 \phi_j(s,t)d x(s,t)
\]
for all $x\in C_0(Q)$ for which the limit exists,
where $(\cdot,\cdot)_2$ denotes the $L_2(Q)$-inner product.
We state some useful facts about the PWZ stochastic integral.
\begin{itemize}
\item[(i)] 
For each $v\in L_2(Q)$, the limit defining the PWZ stochastic integral  $\langle{v,x}\rangle$ exists 
for s-a.e. $x\in C_0(Q)$, and this limit 
is essentially independent of the choice of the complete  orthonormal set  $\{\phi_n\}$.
\item[(ii)] 
If $v$ is of bounded variation on $Q$, then the PWZ stochastic integral 
$\langle{v,x}\rangle$  equals the Riemann--Stieltjes integral $\int_0^T\int_0^S v(s,t)dx(s,t)$  
for $m_{\mathrm{y}}$-a.e. $x\in C_0 (Q)$. 
\item[(iii)]   The PWZ   stochastic integral has the expected linearity properties.
That is, for any real number $c$, $v\in L_2(Q)$, and $x\in C_{0}(Q)$,
it follows that 
$\langle{v,cx}\rangle=c\langle{v,x}\rangle=\langle{cv,x}\rangle$.
\item[(iv)]   For each $v\in L_2(Q)$,  $\langle{v,x}\rangle$
is a Gaussian random variable on $C_0(Q)$ with mean zero and variance
$\|v\|_2^2$.
From this, it follows that
\begin{equation}\label{eq:w-int-f-complex}
\int_{C_0(Q)} \exp\big\{i\alpha\langle{v,x}\rangle\big\}dm_{\mathrm{y}}(x)=\exp\bigg\{-\frac{\alpha^2}{2}\|v\|_2^2\bigg\}.
\end{equation}
for each $\alpha \in \mathbb C$.
\item[(v)]
For all $u,v\in L_2(Q)$, it follows that
\[
\int_{C_0(Q)} \langle{u,x}\rangle\langle{v,x}\rangle dm_{\mathrm{y}}(x)
=(u,v)_{2}.
\]
Thus, if  $\{v_{1}, \ldots,v_{n} \}$ is an orthogonal  set in $L_2(Q)$,
then the Gaussian random variables $\langle{v_j,x}\rangle$'s  are independent.
\end{itemize}

\par
Throughout this paper we let
\[
\begin{aligned}
\mathrm{Supp}_2(Q) 
&
=\{ h \in L_2(Q): m_L^2( \mathrm{supp}(h)) = ST\} 
\\&
=\{ h \in L_2(Q): h \ne 0\,\,  m_L^2\mbox{-a.e  on } Q\} 
\end{aligned}
\]
and
\[
\begin{aligned}
\mathrm{Supp}_{BV}(Q) 
&=\{ h:  h \mbox{ is of bounded variation  with }   h \ne 0\,\,  m_L^2\mbox{-a.e  on } Q\} 
\end{aligned}
\]
where $m_L^2$ denotes the Lebesgue measure on $Q$.
Then one can see that 
$\mathrm{Supp}_{BV}(Q)\subset \mathrm{Supp}_2(Q).$

\par
Given a function $h$ in $\mathrm{Supp}_2(Q)$, we   consider  
the stochastic integral $\mathcal Y_{h}(x;s,t)$ given by 
\begin{equation}\label{eq:yprocess}
\mathcal{Y}_h(x;s,t)
=\langle{\chi_{[0,s]\times[0,t]} h ,x}\rangle,
\end{equation}
for $x\in C_0(Q)$ and $(s,t)\in Q$.
Then the process $\mathcal{Y}_h$  
on $C_0(Q)\times Q$
is a Gaussian process with mean zero  and covariance function
\[
\int_{C_0(Q)}
\mathcal Y_{h}(x;s,t)\mathcal{Y}_h(x;s',t')dm_{\mathrm{y}}(x)
=\int_0^{\min\{t,t'\}}\int_0^{\min\{s,s'\}} h^2 (\nu,\tau ) d\nu d\tau.
\]
Furthermore one can see that
\begin{equation}\label{eq:coco-rel}
\begin{aligned}
&\int_{C_0(Q)}
\mathcal Y_{h_1}(x;s,t)\mathcal{Y}_{h_2}(x;s',t')dm_{\mathrm{y}}(x)\\
&
=\int_0^{\min\{t,t'\}}\int_0^{\min\{s,s'\}} h_1 (\nu,\tau)h_2 (\nu,\tau)d\nu d\tau.
\end{aligned}
\end{equation}

Since the covariance function of $\mathcal Y_h (x; \cdot,\cdot)$ is stochastically continuous, 
we may assume that almost every sample path of $\mathcal Y_h (x; \cdot,\cdot)$ is in $C_0(Q)$.
Also, if $h$ is  a function in $\mathrm{Supp}_{BV}(Q)$, then for all $x\in C_0(Q)$, 
$\mathcal Y_h(x;s,t)$ is continuous in $(s,t)\in Q$,  and so  $\mathcal Y_h(x;\cdot,\cdot)$ is in 
$C_0(Q)$.
Thus, for the definition of the generalized analytic Yeh--Feynman integral 
of functionals on  $C_0(Q)$, we require $h$ to be in  $\mathrm{Supp}_{BV}(Q)$ rather than
simply in  $\mathrm{Supp}_{2}(Q)$.

\section{Generalized analytic Yeh--Feynman integral}

\par
In this section we introduce the generalized analytic Yeh--Feynman integral
of functionals on $C_0(Q)$.
We then provide a class of  generalized Yeh--Feynman integrable functionals.

\par  
Throughout the rest of this paper, let $\mathbb C$, $\mathbb C_+$ and $\widetilde{\mathbb C}_+$
denote  the complex numbers, the  complex numbers with positive real part, 
and the  nonzero complex numbers with nonnegative real part,  respectively.

\par
Given a  Gaussian process  $\mathcal Y_h$ with $h\in \mathrm{Supp}_{BV}(Q)$,
we define the (generalized) $\mathcal Y_h$-Yeh--Wiener integral
(namely, the Yeh--Wiener integral associated with the Gaussian paths  $\mathcal Y_h(x; \cdot,\cdot)$)
for functionals  $F$ on $C_0(Q)$ by the formula
\[
I_{h}[F]\equiv I_{h,x}[F(\mathcal Y_h(x;\cdot,\cdot))]
\equiv \int_{C_0(Q)} F(\mathcal Y_h(x;\cdot,\cdot))dm_{\mathrm{y}}(x).
\] 

Let $F:C_0(Q)\to\mathbb C$ be a scale-invariant measurable functional such that
\[
J_F(h;\lambda)
=I_{h}[F(\lambda^{-1/2}\cdot )]
\equiv I_{h,x}[F(\lambda^{-1/2}\mathcal Y_h(x;\cdot,\cdot))]
\]
exists as a finite number for all $\lambda>0$. If there exists a function $J_F^*(h;\cdot)$
analytic on $\mathbb C_+$ such that $J_F^*(h;\lambda)=J_F(h;\lambda)$  for all $\lambda>0$, 
then $J_F^*(h;\lambda)$ is defined to be the    analytic  $\mathcal Y_h$-Yeh--Wiener integral 
(namely,  the analytic  Yeh--Wiener integral  associated with Gaussian paths $\mathcal Y_h(x;\cdot,\cdot)$) 
of $F$ over $C_0(Q)$ with parameter $\lambda$. For $\lambda\in \mathbb C_+$ we write
\[ 
\begin{aligned}
I_h^{\mathrm{an.}{yw}_{\lambda}}[F]
&
\equiv I_{h,x}^{\mathrm{an.}{yw}_{\lambda}}[F(\mathcal Y_h(x;\cdot,\cdot))] \\
&
 \equiv  \int_{C_0(Q)}^{\mathrm{an.}{yw}_{\lambda}}F (\mathcal Y_h(x;\cdot,\cdot))d m_{\mathrm{y}}(x)
=J_F^*(h;\lambda).
\end{aligned}
\] 
Let $q\ne 0$  be a real number, and let $F$ be a scale-invariant measurable functional 
whose analytic  $\mathcal Y_h$-Yeh--Wiener integral   $I_h^{\mathrm{an.}{yw}_{\lambda}}[F]$
exists for all $\lambda\in\mathbb C_+$. If the following limit exists, we call it the  
analytic  $\mathcal Y_h$-Yeh--Feynman integral  (namely, the  analytic  Yeh--Feynman integral associated 
with Gaussian paths $\mathcal Y_h(x;\cdot,\cdot)$)  of $F$ with parameter $q$, and we write 
\[
\begin{aligned}
I_h^{\mathrm{an.}{yf}_{q}}[F]
&
\equiv I_{h,x}^{\mathrm{an.}{yf}_{q}}[F(\mathcal Y_h(x;\cdot,\cdot))]
\equiv \int_{C_0(Q)}^{\mathrm{an.}{yf}_{q}}F (\mathcal Y_h(x;\cdot,\cdot))d m_{\mathrm{y}}(x)\\
&
=\lim_{\lambda\to-iq}I_{h,x}^{\mathrm{an.}{yw}_{\lambda}}[F(\mathcal Y_h(x;\cdot,\cdot))]
\end{aligned}
\]
where $\lambda$  approaches $-iq$ through values in $\mathbb C_+$.

\par
Let $\mathcal M(L_2(Q))$ be the space of complex-valued, countably additive Borel measures on $\mathcal B(L_2(Q))$, 
the Borel class of $C_0(Q)$. Then the measure $f$ in $\mathcal M(L_2(Q))$  necessarily has finite total variation $\|f\|$, 
and $\mathcal M(C_0(Q))$ is a Banach algebra under the norm $\|\cdot\|$ and with convolution as multiplication, see  \cite{Cohn,Rudin}.
 The Banach algebra $\mathcal S(L_2(Q))$ consists of those functionals expressible in
the form
\begin{equation}\label{eq:element}
F(x)=\int_{L_2(Q)}\exp\{i\langle{u,x}\rangle\}df(u)
\end{equation}
for s-a.e. $x$ in $C_0(Q)$ where the associated measure  $f$ is an element of $\mathcal M(L_2(Q))$.
For a more detailed study of the Banach algebra  $\mathcal S(L_2(Q))$, see \cite{ACY,CAC88}.

\par
The following lemma, which follows quite easily from the definition of the PWZ stochastic integral, 
plays a key role in this paper.

\renewcommand{\thelemma}{\thesection.1}
\begin{lemma}
For each $\alpha\in L_2(Q)$ and each $h\in \mathrm{Supp}_{BV}(Q) $,
\begin{equation}\label{eq:pwz-basic}
\langle{\alpha,\mathcal Y_h(x;\cdot,\cdot)}\rangle
=\langle{\alpha h, x}\rangle
\end{equation}
for s-a.e. $x\in C_0(Q)$.
\end{lemma}

\renewcommand{\thelemma}{\thesection.2}
\begin{lemma} 
Let $F \in \mathcal S(L_2(Q))$ be given by \eqref{eq:element}
and let $h$ be a function in $\mathrm{Supp}_{BV}(Q)$. Then the functional $F_h$ given by
$F_h(x)=F(\mathcal Y_h(x;\cdot,\cdot))$ belongs to the Banach algebra $\mathcal S(L_2(Q))$.
\end{lemma}
\begin{proof}
Let $\Phi_h: L_2(Q)\to L_2(Q)$ be given by $\Phi_h(u)=uh$, the  pointwise multiplication of $u$ and $h$ in $L_2(Q)$.
Then $\Phi_h$ is easily seen to be continuous and so is Borel measurable.
Hence $f_{\Phi_h}\equiv f\circ \Phi_h^{-1}$ is in $\mathcal M(L_2(Q))$.
In addition, for each $\rho>0$, using the change of variables theorem \cite[p.163]{Halmos}
and \eqref{eq:pwz-basic}, it follows that   for a.e. $x$ in $C_0(Q)$,
\[
\begin{aligned}
 \int_{L_2(Q)}\exp\{i\rho \langle{u,x}\rangle\}df_{\Phi_h}(u) 
&=\int_{L_2(Q)}\exp\{i\rho \langle{u,x}\rangle\}d(f\circ  \Phi_h^{-1})(u)\\
&=\int_{L_2(Q)}\exp\{i\rho \langle{\Phi_h(u),x}\rangle\}d f (u)\\
&=\int_{L_2(Q)}\exp\{i\rho \langle{u h,x}\rangle\}d f (u)\\
&=\int_{L_2(Q)}\exp\{i\rho \langle{u  ,\mathcal Y_h(x;\cdot,\cdot)}\rangle\}d f (u)\\
&= F(\rho \mathcal Y_h(x;\cdot,\cdot))\\
&=F_h(\rho x)
\end{aligned}
\]
as desired.
 \qed\end{proof}

\par
We first provide the existence theorem of the  generalized 
analytic  Yeh--Feynman integral of the functionals in $\mathcal  S(L_2(Q))$.

\renewcommand{\thetheorem}{\thesection.3}
\begin{theorem} \label{prop:F-int}
Let $F \in \mathcal  S(L_2(Q))$ be given by   \eqref{eq:element}.
Then for all  $h\in \mathrm{Supp}_{BV}(Q)$ and any nonzero real number  $q$,  
the analytic  $\mathcal Y_h$-Yeh--Feynman  integral, $I_h^{\mathrm{an.}{yf}_{q}}[F]$  of $F$ 
exists and is given by the formula
\begin{equation}\label{eq:Feynman-int}
I_h^{\mathrm{an.}{yf}_{q}}[F]
=\int_{L_2(Q)}\exp\bigg\{-\frac{i}{2q} \|u h\|_2^2  \bigg\} df(u).
\end{equation}
\end{theorem}
\begin{proof}
Using \eqref{eq:element}, the usual Fubini theorem, \eqref{eq:pwz-basic},  
and  \eqref{eq:w-int-f-complex}, it follows  that for all $\lambda>0$,
\[
\begin{aligned}
J_F(h;\lambda) 
&=\int_{C_0(Q)} F(\lambda^{-1/2}\mathcal Y_h(x;\cdot,\cdot))dm_{\mathrm{y}}(x)\\
&=\int_{L_2(Q)}\bigg[\int_{C_0(Q)} \exp\big\{i\lambda^{-1/2}\langle{u h,x}\rangle  
\big\}d m_{\mathrm{y}} \bigg]df(u) \\
&= \int_{L_2(Q)}  \exp\bigg\{-\frac{1}{2\lambda} \|u h\|_2^2 \bigg\}  df(u).
\end{aligned}
\]
Now let 
\[
J_F^*(h;\lambda)=\int_{L_2(Q)}\exp\bigg\{-\frac{1}{2\lambda}\|u h\|_2^2\bigg\}df(u)
\]
for $\lambda \in \mathbb C_+$. Then $J_F^*(h;\lambda)=J_F(h;\lambda)$ 
for all $\lambda>0$ and 
\[
|J_F^*(h;\lambda)|
\le 
\int_{L_2(Q)}\bigg|\exp\bigg\{-\frac{\|u h\|_2^2}{2\lambda}\bigg\}\bigg|d|f|(u)
\le 
\int_{L_2(Q)}d|f|(u)= \|f\|<+\infty
\] 
for all $\lambda\in\mathbb C_+$, since $\textrm{Re}(1/\lambda)>0$. 
Thus,  applying the dominated convergence theorem,  
we see that  $J_F^*(h;\lambda)$  is continuous on  $\widetilde{\mathbb C}_+$.  Also, because 
$\phi(\lambda)\equiv  \exp\{- \|u h\|_2^2/(2\lambda)  \}$
is analytic on $\mathbb C_+$, applying the usual Fubini theorem and the Cauchy integration theorem  
it follows that
\[
\int_{\triangle} J_F^*(h;\lambda) d \lambda
=\int_{L_2(Q)} \int_{\triangle}\phi(\lambda) d \lambda df(u)=0
\]
for all rectifiable simple closed curve $\triangle$ lying in $\mathbb C_+$. 
Thus by the Morera theorem,  $J_F^*(h;\lambda)$ is  analytic on $\mathbb C_+$.
Therefore the analytic $\mathcal Y_h$-Yeh--Wiener  integral 
$I_h^{\mathrm{an},yw_{\lambda}}[F]=J_F^*(h;\lambda)$
exists. 
Finally, applying the dominated convergence theorem it follows  that  
$I_h^{\mathrm{an.}{yf}_{q}}[F]=\lim\limits_{
\substack{\lambda\to -iq \\  \lambda\in \mathbb C_+}}I_h^{\mathrm{an.}{yw}_{\lambda}}[F]$
is given by the right-hand  side of \eqref{eq:Feynman-int}.
 \qed\end{proof}

\section{Fubini theorems for the generalized analytic Yeh--Feynman integral}

\par
In this  section we study Fubini theorems  for the iterated  $\mathcal Y_h$-Yeh--Feynman 
integrals. In \cite{HSS01-F}, Huffman, Skoug and Storvick presented   a Fubini 
theorem involving the iterated analytic Feynman  integrals for  functionals on 
the classical Wiener space $C_0[0,T]$.
The Fubini theorem  can be extended to the $\mathcal Y_h$-Yeh--Feynman 
integral  on the Yeh--Wiener space $C_0(Q)$ as follows:

\renewcommand{\thetheorem}{\thesection.1}
\begin{theorem}\label{multi-single-q}
Let $F\in \mathcal  S(L_2(Q))$  be given by equation \eqref{eq:element}  
and let $\{q_1,q_2,\ldots$, $q_n\}$ be a set of nonzero real numbers with 
\[ 
\frac{1}{q_1}+\frac{1}{q_2}+\cdots+\frac{1}{q_k}\ne0
\]
for each $k\in\{2,\ldots, n\}$. Then for any   function $h$ in $\mathrm{Supp}_{BV}(Q)$,
it follows that
\begin{equation}\label{eq:Q-iter} 
\begin{aligned}
&
I_{h,x_{n}}^{\mathrm{an.}{yf}_{q_n}}\bigg[I_{h,x_{n-1}}^{\mathrm{an.}{yf}_{q_{n-1}}}\bigg[
   \cdots\bigg[I_{h,x_{2}}^{\mathrm{an.}{yf}_{q_2}}\bigg[I_{h,x_{1}}^{\mathrm{an.}{yf}_{q_1}}
\bigg[F\bigg(\sum_{j=1}^n\mathcal Y_h(x_j,\cdot)\bigg)\bigg]\bigg]\bigg]\cdots\bigg]\bigg] \\
&
=I_{h,x}^{\mathrm{an.}{yf}_{\alpha_n}}[F(\mathcal Y_h(x;\cdot,\cdot))] ,
\end{aligned}
\end{equation}
where 
\begin{equation}\label{eq:Q-iter-parameter}
\alpha_n=\bigg(\frac{1}{q_1}+\frac{1}{q_2}+\cdots+\frac{1}{q_n}\bigg)^{-1}.
\end{equation}
\end{theorem}

\par
Equation  \eqref{eq:Q-iter} 
tells us that an iterated  analytic $\mathcal Y_h$-Yeh--Feynman integral can be reduced to a single  
analytic $\mathcal Y_h$-Yeh--Feynman integral.
In  this section  we establish that the iterated  generalized Yeh--Feynman integrals associated 
with different Gaussian processes  also can be reduced to a single  generalized  Yeh--Feynman integral.

\par
To obtain our Fubini theorems 
for the Yeh--Feynman integrals associated with Gaussian processes (see Theorem \ref{thm:iter-gfft} below) 
we adopt the following conventions.
Let $h_1$ and $h_2$ be nonzero functions in $L_2(Q)$. 
Then there exists a nonzero function  $\mathbf{s}$  in $L_2(Q)$ 
such that
\begin{equation}\label{eq:fn-rot}
\mathbf{s}^2(s,t)=h_1^2(s,t)+h_2^2(s,t)
\end{equation}
for $m_L^2$-a.e. $(s,t)\in Q$.
Note that  the function `$\mathbf{s}$' satisfying \eqref{eq:fn-rot} 
is not unique. We will use the symbol  $\mathbf{s}(h_1,h_2)$ for the functions 
`$\mathbf{s}$' that satisfy \eqref{eq:fn-rot} above. Inductively, given a 
set  $\mathcal H=\{h_1,\ldots, h_n\}$  of nonzero functions in $L_2(Q)$, 
let  
\[
\mathbf{s}(\mathcal H)\equiv \mathbf{s}(h_1,h_2,\ldots,h_n)
\]
be the set of functions $\mathbf{s}$ which satisfy the relation 
\begin{equation}\label{eq:fn-rot-ind}
\mathbf{s}^2(s,t)=h_1^2(s,t)+\cdots+h_n^2(s,t) 
\end{equation}
for $m_L^2$-a.e. $(s,t)\in Q$.
We also  note that if  the functions $h_1,\ldots, h_n$ are in $\mathrm{Supp}_{BV}(Q)$,
then we can take  $\mathbf{s}(\mathcal H)$ to be in $\mathrm{Supp}_{BV}(Q) $.
By an induction argument we see that 
\[\label{eq:link}
\mathbf{s}(\mathbf{s}(h_1,h_2,\ldots,h_{k-1}),h_k)
=\mathbf{s}(h_1,h_2,\ldots,h_k) 
\]
for all $k\in\{2,\ldots,n\}$.

\par
In our next lemma we obtain a Fubini theorem for the iterated Yeh--Wiener 
integral  of functionals $F$  in $\mathcal S(L_2(Q))$. 

\renewcommand{\thelemma}{\thesection.2}
\begin{lemma} \label{lemma-funini-basic}
Let $h_1$ and $h_2$ be   functions in $\mathrm{Supp}_{BV}(Q)$ 
and let $F\in \mathcal S(L_2(Q))$  be given by equation \eqref{eq:element}. 
Then for all $\alpha$ and  $\beta$  in $\mathbb R$,
\begin{equation}\label{eq:Fubini-initial}
\begin{aligned}
&\int_{C_0(Q)}\bigg[\int_{C_0(Q)}
F(\alpha \mathcal Y_{h_1}(x_1;\cdot,\cdot)+\beta \mathcal Y_{h_2}(x_2;\cdot,\cdot)) 
dm_{\mathrm{y}}(x_1)\bigg]  dm_{\mathrm{y}}(x_2)\\
&=\int_{C_0(Q)}\bigg[\int_{C_0(Q)}
F(\alpha \mathcal Y_{h_1}(x_1;\cdot,\cdot)+\beta \mathcal Y_{h_2}(x_2;\cdot,\cdot)) 
 dm_{\mathrm{y}}(x_2)\bigg]  dm_{\mathrm{y}}(x_1). 
\end{aligned} 
\end{equation}
In  addition, both expressions in \eqref{eq:Fubini-initial} are 
given by the expression
\begin{equation}\label{eq:Fubini-initial-evalu}
\int_{L_2(Q)}
\exp\bigg\{ - \frac{\alpha^2}{2}\|uh_1\|_2^2 
 - \frac{\beta^2}{2}\|uh_2\|_2^2 \bigg\} df(u).
\end{equation}
\end{lemma}
\begin{proof}
Using \eqref{eq:element} and \eqref{eq:pwz-basic},     it follows that
\[
\int_{C_0(Q)}| F(\rho \mathcal Y_h(x;\cdot,\cdot))| d m_{\mathrm y}(x)
 \le\int_{C_0(Q)}\|f \| d m_{\mathrm y}(x) 
 =\|f \| 
 < +\infty 
\]
for each $\rho>0$.
Hence by the usual Fubini theorem, 
we have  equation \eqref{eq:Fubini-initial} above. 
Furthermore,  using the usual Fubini theorem, \eqref{eq:pwz-basic},
and  \eqref{eq:w-int-f-complex}, 
it follows that for all $\alpha$ and $\beta$  in $\mathbb R$,
\[
\begin{aligned}
&\int_{C_0(Q)} \bigg[\int_{C_0(Q)} 
F(\alpha \mathcal Y_{h_1}(x_1;\cdot,\cdot)+\beta \mathcal Y_{h_2}(x_2;\cdot,\cdot))
dm_{\mathrm y}(x_1) \bigg]dm_{\mathrm y}(x_2)\\
&=\int_{L_2(Q)} 
\bigg[\int_{C_0(Q)}\exp\big\{i\alpha\langle{uh_1,x_1}\rangle\big\}dm_{\mathrm y}(x_1)\bigg]\\
&\qquad\quad\times
\bigg[\int_{C_0(Q)}\exp\big\{i\beta\langle{uh_2, x_2}\rangle\big\}dm_{\mathrm y}(x_2)\bigg]df(u)\\
&=\int_{L_2(Q)}  
\exp \bigg\{ -\frac{\alpha^2}{2}\|uh_1\|_2^2\bigg\}
\exp \bigg\{ -\frac{\beta^2}{2}\|uh_2\|_2^2\bigg\}df(u)\\
&=\int_{L_2(Q)} \exp \bigg\{ -\frac{\alpha^2}{2}\|uh_1\|_2^2 
-\frac{\beta^2}{2}\|uh_2\|_2^2\bigg\}df(u),
\end{aligned}   
\]
as desired.
 \qed\end{proof}

\renewcommand{\thetheorem}{\thesection.3}

\begin{theorem} \label{thm:iter-gfft}
Let $h_1$,  $h_2$, and  $F$  be as in Lemma \ref{lemma-funini-basic}.
Then,  for any nonzero real number $q$, the iterated  Yeh--Feynman integral, 
$I_{h_2}^{\mathrm{an.}{yf}_{q}}[I_{h_1 }^{\mathrm{an.}{yf}_{q}}[F]]$
of $F$ exists and is given by the formula
\begin{equation}\label{eq:fubini-01}
\begin{aligned}
&
I_{h_2,x_2}^{\mathrm{an.}{yf}_{q}}\big[I_{h_1,x_1}^{\mathrm{an.}{yf}_{q}}
\big[F\big(\mathcal Y_{h_1}(x_1;\cdot,\cdot)+\mathcal Y_{h_2}(x_2;\cdot,\cdot)\big)\big]\big] \\
&
=\int_{L_2(Q)}\exp\bigg\{   -\frac{i}{2q}\sum_{j=1}^2\|uh_j\|_2^2\bigg\}df(u).
\end{aligned}
\end{equation}
Furthermore, it follows that
\[
I_{h_2,x_2}^{\mathrm{an.}{yf}_{q}}\big[I_{h_1,x_1}^{\mathrm{an.}{yf}_{q}}
\big[F\big(\mathcal Y_{h_1}(x_1;\cdot,\cdot)+\mathcal Y_{h_2}(x_2;\cdot,\cdot)\big)\big]\big] 
 = I_{\mathbf{s}(h_1,h_2),x}^{\mathrm{an.}{yf}_{q}}    
 \big[F\big(\mathcal Y_{\mathbf{s}(h_1,h_2)}(x;\cdot ,\cdot) \big)\big]
\]
where $\mathbf{s}(h_1,h_2)$ is a function 
in  $\mathrm{Supp}_{BV}(Q)$ 
satisfying  relation \eqref{eq:fn-rot}  above.
\end{theorem}
\begin{proof}
Using \eqref{eq:Fubini-initial} together with  \eqref{eq:Fubini-initial-evalu}, 
it follows  that  for all $(\lambda_1,\lambda_2)\in (0,+\infty)\times(0,+\infty)$,
\[
\begin{aligned}
&
I_{h_2,x_2}\big[I_{h_1,x_1}\big[F\big(\lambda_1^{-1/2} \mathcal Y_{h_1}(x_1;\cdot,\cdot)
+\lambda_2^{-1/2} \mathcal Y_{h_2}(x_2;\cdot,\cdot)\big)\big]\big]\\
&
=\int_{L_2(Q)}
\exp\bigg\{ - \frac{1}{2\lambda_1}\|uh_1\|_2^2 
 - \frac{1}{2\lambda_2}\|uh_2\|_2^2  \bigg\} df(u).
\end{aligned}
\]
For each $\lambda_2>0$ it can be analytically continued in $\lambda_1$
for $\lambda_1\in \mathbb C_+$, and   for each $\lambda_1>0$ it also can be analytically 
continued in $\lambda_2$ for $\lambda_2\in \mathbb C_+$, 
because for any $(\lambda_1,\lambda_2)\in\mathbb C_+\times\mathbb C_+$,
\[
\begin{aligned}
&
\bigg|\int_{L_2(Q)}
\exp\bigg\{ - \frac{1}{2\lambda_1}\|uh_1\|_2^2 
 - \frac{1}{2\lambda_2}\|uh_2\|_2^2 \bigg\} df(u)\bigg|\\
&
\le\int_{L_2(Q)}\bigg|
\exp\bigg\{ - \sum_{j=1}^2\frac{\mathrm{Re}(\lambda_j)-i\mathrm{Im}(\lambda_j)}{2|\lambda_j|^2}\|uh_j\|_2^2  \bigg\} \bigg|d|f|(u) 
 \le \|f\| 
<+\infty.
\end{aligned}
\]
 Thus we obtain the analytic continuation 
\[
\begin{aligned}
&I_{h_2,x_2}^{\mathrm{an.}{yw}_{\lambda_2}}[I_{h_1,x_1}^{\mathrm{an.}{yw}_{\lambda_1}}
\big[F\big( \mathcal Y_{h_1}(x_1;\cdot,\cdot)
+ \mathcal Y_{h_2}(x_2;\cdot,\cdot)\big)\big]\big]\\
& 
=\int_{L_2(Q)}
\exp\bigg\{ - \frac{1}{2\lambda_1}\|uh_1\|_2^2 
 - \frac{1}{2\lambda_2}\|uh_2\|_2^2 \bigg\} df(u)
\end{aligned}
\]
of $I_{h_2,x_2} [I_{h_1,x_1}[F(\lambda_1^{-1/2} \mathcal Y_{h_1}(x_1;\cdot,\cdot)
+\lambda_2^{-1/2} \mathcal Y_{h_2}(x_2;\cdot,\cdot))]]$
as a function of $(\lambda_1,\lambda_2)\in \mathbb C_+\times \mathbb C_+$,
and so it follows that
\[
\begin{aligned}
& I_{h_2,x_2}^{\mathrm{an.}{yf}_{q}}\big[I_{h_1,x_1}^{\mathrm{an.}{yf}_{q}}
\big[F \big( \mathcal Y_{h_1}(x_1;\cdot,\cdot)
+ \mathcal Y_{h_2}(x_2;\cdot,\cdot)\big)\big]\big] \\
& =\lim_{\substack{\lambda_2\to -iq\\ \lambda_2\in\mathbb C_+} }
\int_{C_0(Q)}\bigg[\lim_{\substack{\lambda_1\to -iq\\ \lambda_1\in\mathbb C_+} }\int_{C_0(Q)}
F\big(\lambda_1^{-1/2}\mathcal Y_{h_1}(x_1;\cdot,\cdot) \\
&\qquad\qquad\qquad\qquad\qquad\qquad
+\lambda_2^{-1/2} \mathcal Y_{h_2}(x_2;\cdot,\cdot)\big) 
dm_{\mathrm{y}}(x_1)\bigg]  dm_{\mathrm{y}}(x_2) \\
&= \int_{L_2(Q)}\exp\bigg\{   -\frac{i}{2q}\sum_{j=1}^2\|uh_j\|_2^2\bigg\}df(u).
\end{aligned}
\]
Next using \eqref{eq:fn-rot}, we observe that
\[
\begin{aligned}
\sum_{j=1}^2\|uh_j\|_2^2
&=\int_0^T\int_0^S u^2(s,t)h_1^2(s,t)dsdt 
+\int_0^T\int_0^S u^2(s,t)h_2^2(s,t)dsdt\\
&=\int_0^T\int_0^S u^2(s,t)\big(h_1^2(s,t)+  h_2^2(s,t)\big)dsdt\\
&=\int_0^T\int_0^S u^2(s,t)\mathbf{s}^2(h_1, h_2) (s,t) dsdt\\
&=\|u\mathbf{s}(h_1, h_2)\|_2^2.
\end{aligned}
\]
Using this and equation \eqref{eq:Feynman-int} with $h$ replaced with
$\mathbf{s}(h_1,h_2)$, the generalized analytic Yeh--Feynman integral 
$I_{\mathbf{s}(h_1,h_2)}^{\mathrm{an.}{yf}_{q}}[F]$ is given by the 
right-hand side of equation \eqref{eq:fubini-01}.
This completes the proof.
 \qed\end{proof}

\par
Using  an induction argument  we  obtain the following 
corollary.

\renewcommand{\thecorollary}{\thesection.4}

\begin{corollary}  \label{coro:ind01}
Let $\mathcal H=\{h_1,\ldots,h_n\}$ be  a set of    functions in 
$\mathrm{Supp}_{BV}(Q)$ and let $F\in \mathcal S(L_2(Q))$  be given by equation 
\eqref{eq:element}. Then,  for any nonzero real number  $q$, 
the iterated  analytic Yeh--Feynman integral in the following equation exist,
 and is given by the formula 
\[
\begin{aligned}
&
I_{h_n,x_n}^{\mathrm{an.}{yf}_{q}}\bigg[\cdots \bigg[I_{h_2,x_2}^{\mathrm{an.}{yf}_{q}}
\bigg[I_{h_1,x_1}^{\mathrm{an.}{yf}_{q}}\bigg[
F\bigg(\sum_{j=1}^n \mathcal Y_{h_j}(x_j;\cdot,\cdot)\bigg)\bigg]\bigg]\bigg]\cdots\bigg]\\
&
= \int_{L_2(Q)}\exp\bigg\{   -\frac{i}{2q}\sum_{j=1}^n\|uh_j\|_2^2\bigg\}df(u).
\end{aligned}
\]
Moreover it follows that
\begin{equation}\label{eq:applicable-thing}
\begin{aligned}
&
I_{h_n,x_n}^{\mathrm{an.}{yf}_{q}}\bigg[\cdots \bigg[I_{h_2,x_2}^{\mathrm{an.}{yf}_{q}}
\bigg[I_{h_1,x_1}^{\mathrm{an.}{yf}_{q}}\bigg[
F\bigg(\sum_{j=1}^n \mathcal Y_{h_j}(x_j;\cdot,\cdot)\bigg)\bigg]\bigg]\bigg]\cdots\bigg]\\
&
=I_{\mathbf{s}(\mathcal H),x}^{\mathrm{an.}{yf}_{q}}
\big[F\big( \mathcal Y_{\mathbf{s}(\mathcal H)}(x;\cdot,\cdot) \big)\big] 
\end{aligned}
\end{equation}
where $\mathbf{s}(\mathcal H)\equiv \mathbf{s}(h_1,\ldots,h_n)$  
is a function in $\mathrm{Supp}_{BV}(Q)$
satisfying  relation  \eqref{eq:fn-rot-ind} above.
\end{corollary}

\renewcommand{\theexample}{\thesection.5}

\begin{example}
Let $\mathcal H_4=\{h_1,h_2,h_3,h_4\}$  be a set of  functions in $\mathrm{Supp}_{BV}(Q)$, where 
\[
\begin{cases}
h_1(s,t)=\sin^2 s \cos t,\\
h_2(s,t)=\sin  s \cos s \cos t,\\
h_3(s,t)=\sin  s \sin t \cos t,\\ 
h_4(s,t)=\sin  s  \cos^2 t\\
\end{cases} 
\]
for $(s,t)\in Q$. In this case, we can choose the function $\mathbf{s}(\mathcal H_4)\equiv \mathbf{s}(h_1,h_2,h_3,h_4)$ to be
\[
\mathbf{s}(\mathcal H_4)\equiv \mathbf{s}(h_1,h_2,h_3,h_4)=\sqrt{2}\sin s \cos t,
\]
since
\[
\sum_{j=1}^4h_j^2(s,t)= 2 \sin^2 s \cos^2 t.
\]
Thus, using \eqref{eq:applicable-thing}, it follows that
 \[
\begin{aligned}
&
I_{h_4,x_4}^{\mathrm{an.}{yf}_{q}}
\bigg[I_{h_3,x_3}^{\mathrm{an.}{yf}_{q}} \bigg[I_{h_2,x_2}^{\mathrm{an.}{yf}_{q}}
\bigg[I_{h_1,x_1}^{\mathrm{an.}{yf}_{q}}
\bigg[F\bigg(\sum_{j=1}^4 \mathcal Y_{h_j}(x_j;\cdot,\cdot)\bigg)\bigg]\bigg]\bigg]\bigg]\\
&
= I_{\mathbf{s}(\mathcal H_4),x}^{\mathrm{an.}{yf}_{q}}\big[F\big(\mathcal Y_{\mathbf{s}(\mathcal H_4)}(x;\cdot,\cdot)\big)\big].
\end{aligned}
\]
\end{example}

\section{Generalized Fourier--Yeh--Feynman transforms and generalized convolution product}

\par
The concept of an $L_1$  analytic   Fourier--Feynman transform  was
introduced by Brue in \cite{Brue}. In \cite{CS76}, Cameron and Storvick introduced an
$L_2$  analytic  Fourier--Feynman transform. In \cite{JS79-a}, Johnson and Skoug developed an $L_p$  
analytic  Fourier--Feynman transform for $1 \le p \le 2$ which extended the results 
in \cite{Brue,CS76}  and gave various
relationships between the $L_1$  and   $L_2$  theories.
The transforms studied in \cite{Brue,CS76,JS79-a} are defined on various classes of  functionals
$F$  on the classical Wiener space.

\par
In this section we apply the Fubini theorems obtained in the previous section
to study several relevant behaviors of the GFYFT 
of functionals  on  Yeh--Wiener space $C_0(Q)$.
In this paper, for simplicity, we restrict our discussion to the case $p =1$; 
however most of our results hold for all  $p \in[1,2]$.

\renewcommand{\thedefinition}{\thesection.1}
\begin{definition}
Let $\mathcal Y_h$ be the Gaussian process given by \eqref{eq:yprocess} with $h\in\mathrm{Supp}_{BV}(Q)$, 
and let $F$ be a scale-invariant measurable functional on $C_0(Q)$. 
For $\lambda\in \mathbb C_+$ and $y\in C_0(Q)$, let 
\[
T_{\lambda,h}(F)(y)=I_{h,x}^{\mathrm{an.}{yw}_{\lambda}}[F(y+\mathcal Y_h(x;\cdot,\cdot))].
\]
Then for $q\in \mathbb R\setminus\{0\}$,  the $L_1$ analytic $\mathcal Y_h$-GFYFT 
(namely, the GFYFT associated with the Gaussian  paths  $\mathcal Y_h(x; \cdot,\cdot)$), $T^{(1)}_{q,h}(F)$ of $F$, 
is defined by the formula  
\[
T^{(1)}_{q,h}(F)(y)= \lim_{\substack{ \lambda\to -iq \\  \lambda\in \mathbb C_+}}T_{\lambda,h}(F)(y)
\]
for s-a.e. $y\in C_{0}(Q)$ whenever this limit exists. That is to say, 
\begin{equation}\label{eq:Feynman-transform}
T^{(1)}_{q,h}(F)(y)=I_{h}^{\mathrm{an.}{yf}_{q}}[F(y+ \cdot )] \equiv I_{h,x}^{\mathrm{an.}{yf}_{q}}[F(y+\mathcal Y_h(x;\cdot,\cdot))] 
\end{equation}
for s-a.e. $y\in C_{0}(Q)$.
\end{definition}

We note that  $T_{q,h}^{(1)}(F)$ exists  and if $F\approx G$, 
then $T_{q,h}^{(1)}(G)$ exists  and  $T_{q,h}^{(1)}(G)\approx T_{q,h}^{(1)}(F)$.
 One can see that for each $h\in L_2(Q)$, $T_{q,h}^{(1)}(F)\approx T_{q,-h}^{(1)}(F)$
since
\begin{equation}\label{meanzero}
\int_{C_0(Q)}F(-x)dm_{\mathrm y}(x)
=\int_{C_0(Q)}F(x)dm_{\mathrm y}(x).
\end{equation}

\par
In view of \eqref{eq:Feynman-transform} and  \eqref{eq:Feynman-int} with  $F$ replaced with $F(y+\cdot)$,
we obtain the following existence theorem.

\renewcommand{\thetheorem}{\thesection.2}

\begin{theorem}\label{thm:gfft}
Let  $F\in\mathcal S(L_2(Q))$ be given by equation \eqref{eq:element}. Then,  
for all  $h\in \mathrm{Supp}_{BV}(Q)$, 
the $L_1$  analytic   $\mathcal Y_h$-GFYFT, $T_{q,h}^{(1)}(F)$ 
of $F$ exists for all nonzero  real  $q$, belongs to $\mathcal S(L_2(Q))$ 
and is given by the formula
\begin{equation}\label{eq:t1q-yw}
T_{q,h}^{(1)}(F)(y)
= \int_{L_2(Q)}\exp\{i\langle{u,y}\rangle\}df_t^h(u)
\end{equation}
for s-a.e. $y\in C_{0}(Q)$, where  $f_t^h$ is the complex measure  
in  $\mathcal M(L_2(Q))$ given by
\[
f_t^{h}(B)=\int_B \exp\bigg\{-\frac{i}{2q}\|uh\|_2^2\bigg\}df(u)
\]
for $B \in \mathcal B(L_2(Q))$.
\end{theorem}

\par
The following corollary is a simple consequence of Theorem \ref {thm:gfft}.

\renewcommand{\thecorollary}{\thesection.3}

\begin{corollary}\label{thm:afft-inverse}
Let   $F$ be  as in Theorem \ref{thm:gfft}.
Then,  
for all  $h\in \mathrm{Supp}_{BV}(Q)$  and all nonzero real   $q$, 
\begin{equation}\label{eq:inverse}
T_{-q, h}^{(1)}\big(T_{q,h}^{(1)}(F)\big)\approx F.
\end{equation}
In other words, the $L_1$ $\mathcal Y_h$-GFYFT, $T_{q,h}^{(1)}$, has the 
inverse transform $\{T_{q,h}^{(1)}\}^{-1}=T_{-q,h}^{(1)}$.
\end{corollary}

\renewcommand{\theremark}{\thesection.4}

\begin{remark}
By Theorem \ref{thm:gfft} and an induction argument,
one can see that for any functional $F$ in $\mathcal S(L_2(Q))$,
 any nonzero real numbers  $q_1$, $q_2$, $\ldots$, $q_n$,
and any nonzero functions $h_1, \ldots, h_n$ in $ \mathrm{Supp}_{BV}(Q)$,
the iterated GFYFT  
\[
T_{q_n,h_n}^{(1)}(T_{q_{n-1},h_{n-1}}^{(1)}(
   \cdots(T_{q_2,h_2}^{(1)}(T_{q_1,h_1}^{(1)}(F)))\cdots))
\]
of $F$ exists and  belongs to $\mathcal S(L_2(Q))$.  
\end{remark}

\par
Next, in \cite{HSS01-I}, Huffman, Skoug and Storvick studied a Fubini 
theorem involving ordinary Fourier--Feynman transform  for  functionals on 
classical Wiener space $C_0[0,T]$.  
Using \eqref{eq:Feynman-transform} and \eqref{eq:Q-iter} with  $F$ replaced with $F(y+\cdot)$,
we also  obtain the following   Fubini theorem involving the $L_1$ analytic GFYFTs of functionals in 
the Banach algebra $\mathcal S(L_2(Q))$.

\renewcommand{\thetheorem}{\thesection.5}

\begin{theorem}\label{multi-single-q-T}
Let $F$ and $\{q_1,q_2, \ldots,q_n\}$ be 
as in Theorem \ref{multi-single-q}. Then it follows that for each  
function $h \in  \mathrm{Supp}_{BV}(Q)$, 
\begin{equation}\label{eq:Q-iter+}
T_{q_n,h}^{(1)}(T_{q_{n-1},h}^{(1)}(
   \cdots(T_{q_2,h}^{(1)}(T_{q_1,h}^{(1)}(F)))\cdots)) (y)
=T_{\alpha_n,h}^{(1)}(F)(y) 
\end{equation}
for s-a.e. $y\in C_0(Q)$,  where $\alpha_n$  is a nonzero real number given by 
\eqref{eq:Q-iter-parameter}.
\end{theorem}

\par
Equation   \eqref{eq:Q-iter+}  above
tells us that  iterated  GFYFT  with different variance parameters $q_1,\ldots,q_n$
 can be reduced to a single  GFYFT.
We will assert    that the composition of   GFYFTs associated 
with different Gaussian processes also can be reduced to a single GFYFT.
In view of \eqref{eq:Feynman-transform} and  \eqref{eq:applicable-thing}  
with  $F$ replaced with $F(y+\cdot)$,
we obtain the following theorem, which  will be very useful to prove our main theorems
in next sections.

\renewcommand{\thetheorem}{\thesection.6}

\begin{theorem}  \label{thm:ind01-2019}
Let $\mathcal H=\{h_1,\ldots,h_n\}$  and let $F\in \mathcal S(L_2(Q))$  
be as in Corollary \ref{coro:ind01}. 
Then it follows that   for all  nonzero real  $q$,     
\begin{equation}\label{eq:gfft-n-fubini}
T_{q,h_n}^{(1)}(T_{q,h_{n-1}}^{(1)}(
   \cdots(T_{q,h_2}^{(1)}(T_{q,h_1}^{(1)}(F)))\cdots))(y)  
=T_{q, \mathbf{s}(\mathcal H)}^{(1)}(F)(y)
\end{equation}
for s-a.e. $y\in C_{0}(Q)$, where 
$\mathbf{s}(\mathcal H)\equiv \mathbf{s}(h_1,\ldots,h_n)$  
is a function in $\mathrm{Supp}_{BV}(Q)$
satisfying  relation  \eqref{eq:fn-rot-ind} above.
\end{theorem}

\renewcommand{\theexample}{\thesection.7}

\begin{example}
 Let $h_1$ and $h_2$ be given by  
\[
h_1(s,t)=\sin \bigg(\frac{2\pi s}{T}\bigg)\sin \bigg(\frac{2\pi t}{T}\bigg)-\cos \bigg(\frac{2\pi s}{T}\bigg)\cos \bigg(\frac{2\pi t}{T}\bigg)
\]
and
\[  
h_2(s, t)=\sin \bigg(\frac{2\pi s}{T}\bigg)\cos  \bigg(\frac{2\pi t}{T}\bigg)+\cos \bigg(\frac{2\pi s}{T}\bigg)\sin \bigg(\frac{2\pi t}{T}\bigg)
\]
on $Q\equiv [0,S]\times[0,T]$, 
respectively. 
Then $h_1$ and $h_2$ are in $\mathrm{Supp}_{BV}(Q)$ and $\mathbf{s}(h_1,h_2)\equiv \pm 1$. Thus, by equation \eqref{eq:gfft-n-fubini}
with $n=2$ and \eqref{meanzero}, we have
\begin{equation}\label{simple-example}
T_{q,h_2}^{(1)}(T_{q,h_1}^{(1)}(F))(y) 
=T_{q,\mathbf{s}(h_1,h_2)}^{(1)}(F)(y) 
\end{equation}
for every $F\in \mathcal S(L_2(Q))$ and s-a.e. $y\in C_{0}(Q)$.  
 \end{example}

\par
We finish this section 
with  a more general assertion for iterated GFFT.
The following corollary follows immediately form  
equations \eqref{eq:gfft-n-fubini} and \eqref{eq:Q-iter+}.

\renewcommand{\thecorollary}{\thesection.8}

\begin{corollary}
Let $\mathcal H_1=\{h_{11},\ldots,h_{1n_1}\}$ and 
$\mathcal H_2=\{h_{21},\ldots,h_{2n_2}\}$ 
be sets of  functions in  $\mathrm{Supp}_{BV}(Q)$ which satisfy the relation
\[
\mathbf{s}(\mathcal H_1)
=\mathbf{s}(\mathcal H_2)
\quad{i.e.,}\quad
\mathbf{s}(h_{11},\ldots,h_{1n_1})
=\mathbf{s}(h_{21},\ldots,h_{2n_2})
\]
for $m_L^2$-a.e. on $Q$. Then  it follows that  for any nonzero real numbers $q_1$ 
and $q_2$ with $q_1+q_2\ne0$,
\[
\begin{aligned}
&T_{q_2,h_{2n_2}}^{(1)}\Big(\cdots\Big(T_{q_2,h_{21}}^{(1)}
\Big(T_{q_1,h_{1n_1}}^{(1)}(\cdots(T_{q_1,h_{11}}^{(1)}(F))\cdots)\Big)\Big)\cdots\Big)(y) \\
&=T_{q_2,h_{2n_2}}^{(1)}\Big(
   \cdots\Big(T_{q_2,h_{21}}^{(1)}\Big(T_{q_1,\mathbf{s}(\mathcal H_1)}^{(1)}(F)\Big)\Big)\cdots\Big)(y) \\
&=T_{q_2,\mathbf{s}(\mathcal H_2)}^{(1)}(T_{q_1,\mathbf{s}(\mathcal H_1)}^{(1)}(F))(y) \\
&=T_{\frac{q_1q_2}{q_1+q_2}, \mathbf{s}(\mathcal H)}^{(1)}(F)(y)
\end{aligned}
\]
for s-a.e. $y\in C_{0}(Q)$, where $\mathcal H$ is a finite  set of  functions 
in $\mathrm{Supp}_{BV}(Q)$ with 
$\mathbf{s}(\mathcal H)=\mathbf{s}(\mathcal H_1)=\mathbf{s}(\mathcal H_2)$.
\end{corollary}

%

\section{Generalized  Fourier--Yeh--Feynman transform   and
generalized convolution product on $C_0(Q)$}\label{sec:aesthetic}

\par 
In this section, as  applications of the results in the previous sections,
we  establish  more general  relationships, such as  \eqref{eq:offt-ocp}  and \eqref{eq:ocp-offt},
between the GFYFT and the GCP on the Yeh--Wiener space $C_0(Q)$. 
 
\par
 The following definition of the GCP on $C_{0}(Q)$
is due to Chang and  Choi  \cite{CC17}.

\renewcommand{\thedefinition}{\thesection.1}

\begin{definition}\label{def:cp}
Let $F$ and $G$ be  scale-invariant measurable functionals on $C_{0}(Q)$.
For $\lambda \in \widetilde{\mathbb C}_+$ and $k_1,k_2\in  \mathrm{Supp}_{BV}(Q)$, 
we define their GCP with respect to $\{\mathcal{Y}_{k_1},\mathcal{Y}_{k_2}\}$ 
(if it exists) by
\[
\begin{aligned}
&
(F*G)_{\lambda}^{(k_1,k_2)}(y)\\
&
=
\begin{cases}
E_{\vec k,x}^{\mathrm{an}.yw_{\lambda}} \big[
F\big(\frac{y+{\mathcal Y}_{k_1} (x,\cdot)}{\sqrt2}\big)
G\big(\frac{y-{\mathcal Y}_{k_2} (x,\cdot)}{\sqrt2}\big)\big],  
                     \,\, &    \lambda \in \mathbb C_+ \\
E_{\vec k,x}^{\mathrm{an}.yf_{q}} \big[ 
F\big(\frac{y+{\mathcal Y}_{k_1} (x,\cdot)}{\sqrt2}\big)
G\big(\frac{y-{\mathcal Y}_{k_2} (x,\cdot)}{\sqrt2}\big)\big],
&\lambda=-iq,\,\, q\in \mathbb R, \,\,q\ne 0.
\end{cases} 
\end{aligned}
\]
When $\lambda =-iq$,  we denote $(F*G)_{\lambda}^{(k_1,k_2)}$ 
by $(F*G)_{q}^{(k_1,k_2)}$.
\end{definition}

 Proceeding as in the proof of \cite[Theorem 3.2]{HPS97-2}, 
we can obtain the  existence of the GCP
of functionals in $\mathcal S(L_2(Q))$.

\renewcommand{\thetheorem}{\thesection.2}

\begin{theorem}  \label{thm:gcp}
Let $k_1$ and $k_2$ be   functions in $\mathrm{Supp}_{BV}(Q)$ and let $F$ 
and $G$ be elements of $\mathcal S(L_2(Q))$ with corresponding finite 
Borel measures $f$ and $g$ in $\mathcal M(L_2(Q))$. Then, the GCP 
$(F*G)_q^{(k_1,k_2)}$ of $F$ and $G$  exists for all  nonzero real  $q$, belongs to 
$\mathcal S(L_2(Q))$, and is given by the formula
\begin{equation}\label{eq:GCP-yw}
(F*G)_q^{(k_1,k_2)}(y)
= \int_{L_2(Q)}\exp\{i\langle{w,y}\rangle\}d\varphi^{k_1,k_2}_c(w)
\end{equation}
for s-a.e. $y\in C_{0}(Q)$, where  
\[
\varphi^{k_1,k_2}_c
=\varphi^{k_1,k_2}\circ\phi^{-1},
\]
$\varphi^{k_1,k_2}$ is the complex measure in $\mathcal M(L_2(Q))$ given 
by
\[
\varphi_{k_1,k_2}(B)
=\int_B \exp\bigg\{-\frac{i}{4q}\|uk_1-vk_2\|_2^2\bigg\}df(u)dg(v)
\]
for $B \in \mathcal B(L_2^2(Q))$, and $\phi: L_2(Q)\times L_2(Q) \to L_2(Q)$ is the 
continuous function  given by $\phi(u,v)=(u+v)/\sqrt2$.
\end{theorem}

 Now we are ready to establish relationships between the GFYFT and the GCP
on $C_0(Q)$.  

\subsection{Relationship I: GFYFT of the  GCP}
 
Our first relationship between  the GFYFT and the GCP shows that the GFYFT of the GCP is a product of GFYFTs.
 To establish the first relationship we need the following lemmas.

\renewcommand{\thelemma}{\thesection.3}

\begin{lemma}\label{lem:process-lemma1}
Given four functions $h_1$, $h_2$, $k_1$, and  $k_2$ in $\mathrm{Supp}_{BV}(Q)$,
let stochastic processes
\[
\mathcal G_{h_1,k_1}, \,\, \mathcal G_{h_2,k_2}: C_0(Q)\times C_0(Q)\times Q \to\mathbb R
\] 
be given by
\[
\mathcal G_{h_1,k_1}(x_1,x_2;s,t)=\mathcal Y_{h_1} (x_1;s,t)+\mathcal Y_{k_1}(x_2;s,t)
\]
and
\[
\mathcal G_{h_2,k_2}(x_1,x_2;s,t)=\mathcal Y_{h_2} (x_1;s,t)-\mathcal Y_{k_2}(x_2;s,t),
\]
respectively.  Then the following assertions are equivalent.
\begin{itemize}
\item[(i)] $\mathcal G_{h_1,k_1}$ and $\mathcal G_{h_2,k_2}$ are independent processes,
\item[(ii)] $h_1h_2=k_1k_2$  in $L_2(Q)$. 
\end{itemize}
\end{lemma}
\begin{proof}
Since the processes $\mathcal G_{h_1,k_1}$ and $\mathcal G_{h_2,k_2}$  
are Gaussian with mean zero,
we know that $\mathcal G_{h_1,k_1}$ and $\mathcal G_{h_2,k_2}$ 
are independent processes if and only if 
\[
\int_{C_0^2(Q)}\mathcal G_{h_1,k_1}(x_1,x_2;s,t) \mathcal G_{h_2,k_2}(x_1,x_2;s',t') d (m_{\mathrm y}\times m_{\mathrm y})(x_1,x_2)=0
\]
for all $(s, t)$ and $(s',t')$  in $Q$.
But, using equation \eqref{eq:coco-rel},  it follows that
\[
\begin{aligned}
&\int_{C_0^2(Q)}\mathcal G_{h_1,k_1}(x_1,x_2;s,t) \mathcal G_{h_2,k_2}(x_1,x_2;s',t') d (m_{\mathrm y}\times m_{\mathrm y})(x_1,x_2)\\
& =\int_{C_0^2(Q)}\Big\{\mathcal Y_{h_1} (x_1;s,t)\mathcal Y_{h_2} (x_1;s',t') \\
& \qquad \qquad \quad
                        -\mathcal Y_{h_1} (x_1;s,t)\mathcal{Y}_{k_2}(x_2;s',t') \\
& \qquad \qquad \quad
     +\mathcal{Y}_{k_1}(x_2;s,t)\mathcal Y_{h_2} (x_1;s',t') \\
& \qquad \qquad \quad
      -\mathcal{Y}_{k_1}(x_2;s,t)\mathcal{Y}_{k_2}(x_2;s',t')\Big\} 
 d (m_{\mathrm y}\times m_{\mathrm y})(x_1,x_2) \\
&=\int_0^{\min\{t,t'\}}\int_0^{\min\{s,s'\}} h_1(\nu,\tau)h_2(\nu,\tau)d\nu d\tau\\
&\,\,\,
-\int_0^{\min\{t,t'\}}\int_0^{\min\{s,s'\}} k_1(\nu,\tau)k_2(\nu,\tau)d\nu d\tau.
\end{aligned}
\]
From this we can obtain the desired result.
 \qed\end{proof}

\renewcommand{\thelemma}{\thesection.4}

\begin{lemma}\label{lem:process-lemma2}
Given  two functions $h$ and   $k$  in $\mathrm{Supp}_{BV}(Q)$,
let   stochastic processes 
\[
\mathcal E_{h,k}^+, \,\, \mathcal E_{h,k}^-: C_0(Q)\times C_0(Q)\times Q \to\mathbb R
\] 
be given by
\[
\mathcal E_{h,k}^+(x_1,x_2;s,t)=\mathcal Y_{h} (x_1;s,t)+\mathcal Y_{k}(x_2;s,t)
\]
and
\[
\mathcal E_{h,k}^-(x_1,x_2;s,t)=\mathcal Y_{h} (x_1;s,t)-\mathcal Y_{k}(x_2;s,t),
\]
respectively.  Then the three processes   $\mathcal E_{h,k}^+$,   $\mathcal E_{h,k}^-$
and $\mathcal Y_{\mathbf{s}(h,k)}$ are mutually equivalent with the normal distribution
$\mathcal N(0, \beta_{h,k}(\cdot,\cdot))$ where  
\[
\beta_{h,k}(s,t)=\int_0^t\int_0^s \mathbf{s}^2(h,k)(\nu,\tau)d\nu d\tau. 
\]
\end{lemma}

\renewcommand{\thetheorem}{\thesection.5}

\begin{theorem} \label{thm:gfft-gcp-compose}
Let $k_1$, $k_2$, $F$,  and $G$ be as in Theorem \ref{thm:gcp},
and let $h$ be a  function in $\mathrm{Supp}_{BV}(Q)$. Assume that $h^2=k_1k_2$ $m_L^2$-a.e. 
on $Q$. Then,  for  all nonzero real  $q$, 
\begin{equation}\label{eq:gfft-gcp}
T_{q,h}^{(1)}\big((F*G)_q^{(k_1,k_2)}\big)(y) 
= T_{q,\mathbf{s}(h,k_1)/\sqrt2}^{(1)}(F)\bigg(\frac{y}{\sqrt2}\bigg)
T_{q,\mathbf{s}(h,k_2)/\sqrt2}^{(1)}(G)\bigg(\frac{y}{\sqrt2}\bigg)  
\end{equation}
for s-a.e. $y\in C_{0}(Q)$, where 
$\mathbf{s}(h,k_j)$'s, $j\in \{1,2\}$, are  functions in $\mathrm{Supp}_{BV}(Q)$ 
which satisfy the relation
\eqref{eq:fn-rot} with $h_1$ and $h_2$ replaced with
$h$ and $k_j$, $j\in \{1,2\}$, respectively. 

In particular, it follows that
\[
T_{q,h}^{(1)} \big((F*G)_{q}^{(h,h)} \big)(y)
=T_{q,h}^{(1)}(F)\bigg(\frac{y}{\sqrt2}\bigg)
 T_{q,h}^{(1)}(G)\bigg(\frac{y}{\sqrt2}\bigg) 
\]
for s-a.e. $y\in C_{0}(Q)$. 
\end{theorem}

\begin{proof}
By Theorems \ref{thm:gcp} and \ref{thm:gfft}, $T_{q,h}^{(1)}((F*G)_q^{(k_1,k_2)})$ belongs to the Banach algebra
$\mathcal S(L_2(Q))$ for all  nonzero real $q$.
Thus, the proof given in \cite[Theorem 3.3]{HPS97-2} with the current hypotheses on $C_0(Q)$ and with Lemmas 
 \ref{lem:process-lemma1} and \ref{lem:process-lemma2} also works here.
 \qed\end{proof}

\renewcommand{\theremark}{\thesection.6}

\begin{remark}
Equation  \eqref{eq:gfft-gcp} above is useful in that it permits 
one to calculate the GFYFT of the GCP of functionals on $C_0(Q)$
without actually calculating the GCP $(F*G)_q^{(k_1,k_2)}$.
\end{remark}

\renewcommand{\theremark}{\thesection.7}

\begin{remark} 
Under the assumptions as given in Theorem \ref{thm:gfft-gcp-compose}, we can prove the
equation \eqref{eq:gfft-gcp} using    \eqref{eq:GCP-yw}, \eqref{eq:t1q-yw}  and
direct calculations, but it is  tedious.
\end{remark}

   Choosing  $h=k_1=k_2\equiv 1$ in  equation \eqref{eq:gfft-gcp},
we have the following relationship
between the ordinary Fourier--Yeh--Feynman transform and the ordinary CP on $C_0(Q)$.

\renewcommand{\thecorollary}{\thesection.8}

\begin{corollary} 
Let $F$ and $G$ be as in Theorem  \ref{thm:gcp}. Then for all  
 real $q\in\mathbb R\setminus\{0\}$,
\[
T_q^{(1)} \big((F*G)_q\big)(y)
=T_q^{(1)}(F)\bigg(\frac{y}{\sqrt2}\bigg)
 T_q^{(1)}(G)\bigg(\frac{y}{\sqrt2}\bigg) 
\]
for s-a.e. $y\in C_{0}(Q)$, where $T_q^{(1)}(F)\equiv T_{q,1}^{(1)}(F)$ denotes the ordinary 
Fourier--Yeh--Feynman transform of $F$ and $(F*G)_q\equiv (F*G)_q^{(1,1)}$ 
denotes the corresponding CP of $F$ and $G$.
\end{corollary}

\par
We now present a  simple example for the assumption in Theorem \ref{thm:gfft-gcp-compose}.
Let  $k_1$, $k_2$ and $h$ be  given by 
\[
\begin{aligned}
k_1(s,t)& =4\sin^2\bigg(\frac{2\pi s}{T}\bigg) \sin^2 \bigg(\frac{2\pi t}{T}\bigg),\\
k_2(s,t)& =4\cos^2\bigg(\frac{2\pi s}{T}\bigg) \cos^2 \bigg(\frac{2\pi t}{T}\bigg)
\end{aligned}
\]
and 
\[
h(t)=\sin\bigg(\frac{4\pi s}{T}\bigg)\sin \bigg( \frac{4\pi t}{T}\bigg)
\]
on $Q$,  respectively. Then it follows that
\[
\begin{aligned}
k_1(s,t)k_2(s,t)
&=16\sin^2\bigg(\frac{2\pi s}{T}\bigg)\cos^2 \bigg(\frac{2\pi s}{T}\bigg)
   \sin^2\bigg(\frac{2\pi t}{T}\bigg)\cos^2 \bigg(\frac{2\pi t}{T}\bigg)\\
&=  \sin^2\bigg(\frac{4\pi s}{T}\bigg)  \sin^2\bigg(\frac{4\pi t}{T}\bigg)\\
&= h^2(s,t).
\end{aligned}
\]
for all $(s,t)\in Q$.

\par
In view of Theorems \ref{thm:ind01-2019} and  \ref{thm:gfft-gcp-compose},
we obtain the following corollary.

\renewcommand{\thecorollary}{\thesection.9}

\begin{corollary} \label{thm:iter-gfft-gcp-compose}
Let $k_1$, $k_2$, $F$ and $G$ be as in Theorem \ref{thm:gcp}
and let $\mathcal H=\{h_1,\ldots,h_n\}$ be  a  set of  functions
in $\mathrm{Supp}_{BV}(Q)$. Assume that  
\[
\mathbf{s}^2(h_1,\ldots,h_n)=k_1k_2
\]
 $m_L^2$-a.e. on $Q$. 
Then,  for   all  nonzero real  $q$, 
\[
\begin{aligned}
&T_{q,h_n}^{(1)}\big(T_{q,h_{n-1}}^{(1)}\big(
   \cdots\big(T_{q,h_2}^{(1)}\big(T_{q,h_1}^{(1)}\big((F*G)_q^{(k_1,k_2)}\big)\big)\big)\cdots\big)\big) (y)\\
&=T_{q,\mathbf{s}(\mathcal H)}^{(1)}\big((F*G)_q^{(k_1,k_2)}\big)(y)\\
&= T_{q,\mathbf{s}(\mathcal K_1)/\sqrt2}^{(1)}(F)\bigg(\frac{y}{\sqrt2}\bigg)
   T_{q,\mathbf{s}(\mathcal K_2)/\sqrt2}^{(1)}(G)\bigg(\frac{y}{\sqrt2}\bigg)  
\end{aligned}
\]
for s-a.e. $y\in C_{0}(Q)$, where $\mathcal K_1=\mathcal H\cup\{k_1\}$
and $\mathcal K_2=\mathcal H\cup\{k_2\}$.
\end{corollary}

 \subsection{Relationship II: GCP of GFYFTs}

Our second relationship between  the GFYFT and the GCP shows that the GCP of the GFYFTs  can be represented as a single GFYFT.

\renewcommand{\thetheorem}{\thesection.10}

\begin{theorem}    \label{thm:cp-tpq02}
Let $k_1$, $k_2$, $F$,  $G$, and $h$ be as in Theorem  \ref{thm:gfft-gcp-compose}.
Then,  for  all nonzero real  $q$, 
\begin{equation}\label{eq:cp-fft-basic}
\Big(T_{q,\mathbf{s}(h,k_1)/\sqrt{2} }^{(1)} (F)*
 T_{q,\mathbf{s}(h,k_2)/\sqrt{2}}^{(1)}(G) \Big)_{-q}^{(k_1,k_2)}(y) 
=T_{q,h}^{(1)} \bigg(F \bigg(\frac{\cdot}{\sqrt2} \bigg)
 G \bigg(\frac{\cdot}{\sqrt2}\bigg)\bigg)(y)
\end{equation}
for s-a.e. $y\in C_0(Q)$, where 
$\mathbf{s}(h,k_j)$'s, $j\in \{1,2\}$, are  functions in $\mathrm{Supp}_{BV}(Q)$ 
which satisfy the relation \eqref{eq:fn-rot} with $h_1$ and $h_2$ replaced with
$h$ and $k_j$, $j\in \{1,2\}$, respectively.

In particular, it follows that
\[
\Big(T_{q,h}^{(1)}(F)*T_{q,h}^{(1)}(G)\Big)_{-q}^{(h,h)} (y)
=T_{q,h}^{(1)}\bigg(F \bigg(\frac{\cdot}{\sqrt2}\bigg)
 G \bigg(\frac{\cdot}{\sqrt2}\bigg) \bigg)(y)
\]
for s-a.e. $y\in C_{0}(Q)$.
\end{theorem}
\begin{proof}
Applying \eqref{eq:inverse}, \eqref{eq:gfft-gcp} with $F$, $G$, and $q$
replaced with $T_{q,\mathbf{s}(h,k_1)/\sqrt{2} }^{(1)} (F)$, 
$T_{q,\mathbf{s}(h,k_2)/\sqrt{2}}^{(1)}(G)$,
and $-q$, respectively, and \eqref{eq:inverse} again, 
it follows that for s-a.e. $y\in C_0(Q)$,
\[
\begin{aligned}
&\Big(T_{q,\mathbf{s}(h,k_1)/\sqrt{2} }^{(1)} (F)*
 T_{q,\mathbf{s}(h,k_2)/\sqrt{2}}^{(1)}(G) \Big)_{-q}^{(k_1,k_2)}(y)\\
&= T_{q,h}^{(1)}\Big(T_{-q,h}^{(1)}\Big(\big(T_{q,\mathbf{s}(h,k_1)/\sqrt{2} }^{(1)} (F)*
 T_{q,\mathbf{s}(h,k_2)/\sqrt{2}}^{(1)}(G) \big)_{-q}^{(k_1,k_2)}\Big)\Big)(y)\\
&= T_{q,h}^{(1)}\bigg(\Big[
T_{-q,\mathbf{s}(h,k_1)/\sqrt2}^{(1)}\Big(T_{q,\mathbf{s}(h,k_1)/\sqrt{2} }^{(1)}(F)\Big)\Big]
\bigg(\frac{\cdot}{\sqrt2}\bigg)\\
&\qquad\,\,\,\,\times
\Big[T_{-q,\mathbf{s}(h,k_2)/\sqrt2}^{(1)}\Big(T_{q,\mathbf{s}(h,k_2)/\sqrt{2} }^{(1)}(G)\Big)\Big]
\bigg(\frac{\cdot}{\sqrt2}\bigg)\bigg)
(y)\\
\end{aligned}
\]
\[
\begin{aligned}
&=T_{q,h}^{(1)} \bigg(F \bigg(\frac{\cdot}{\sqrt2} \bigg)
 G \bigg(\frac{\cdot}{\sqrt2}\bigg)\bigg)(y)
\end{aligned}
\]
as desired.
 \qed\end{proof}

Letting $h=k_1=k_2\equiv 1$ in equation \eqref{eq:cp-fft-basic}, one can see that equation 
\eqref{eq:R2-coro1} below holds.

\renewcommand{\thecorollary}{\thesection.11}

\begin{corollary} 
Let $F$ and $G$ be as in Theorem  \ref{thm:gcp}. Then, for all  
real   $q\in\mathbb R\setminus\{0\}$,
\begin{equation}\label{eq:R2-coro1}
\big( T_q^{(1)}(F)*T_q^{(1)}(G)\big)_{-q} (y)
=T_q^{(1)}\bigg(F \bigg(\frac{\cdot}{\sqrt2}\bigg)
 G \bigg(\frac{\cdot}{\sqrt2}\bigg) \bigg)(y)
\end{equation}
for s-a.e. $y\in C_{0}(Q)$
\end{corollary}

 We next establish two types of extension    of Theorem \ref{thm:cp-tpq02} above.

\renewcommand{\thetheorem}{\thesection.12}

\begin{theorem} \label{thm:iter-gfft-gcp-compose}
Let  $k_1$, $k_2$, $F$, and $G$ be as in 
Theorem \ref{thm:gcp}, and let $\mathcal H=\{h_1,\ldots,h_n\}$ 
be  a finite sequence of    functions in $\mathrm{Supp}_{BV}(Q)$. 
Assume that  
\[
\mathbf{s}^2(\mathcal H)   \equiv \mathbf{s}^2(h_1,\ldots,h_n)=k_1k_2
\] 
for $m_L^2$-a.e. on  $Q$, where $\mathbf{s}(\mathcal H)$ is the function in $\mathrm{Supp}_{BV}(Q)$
satisfying \eqref{eq:fn-rot-ind} above. 
Then,  for  all  nonzero real  $q$, 
\begin{equation} \label{eq:multi-rel-01}
\begin{aligned}
&\Big(T_{q,k_1/\sqrt2}^{(1)}
\big(T_{q,h_n/\sqrt2}^{(1)}\big(\cdots\big(T_{q,h_2/\sqrt2}^{(1)}\big(T_{q,h_1/\sqrt2}^{(1)}(F)\big)\big)\cdots\big)\big)
\\
&\qquad\qquad  
*T_{q,k_2/\sqrt2}^{(1)}\big( T_{q,h_n/\sqrt2}^{(1)}\big( \cdots\big(T_{q,h_2/\sqrt2}^{(1)}\big(T_{q,h_1/\sqrt2}^{(1)}(G)
\big)\big)\cdots\big)\big)\Big)_{-q}^{(k_1,k_2)}(y)\\
&=\Big(T_{q, \mathbf{s}(\mathcal H,k_1)/\sqrt2}^{(1)}(F)
*T_{q, \mathbf{s}(\mathcal H,k_2)/\sqrt2}^{(1)}(G)\Big)_{-q}^{(k_1,k_2)}(y)\\
&= T_{q,\mathbf{s}(\mathcal H)}^{(1)} \bigg(F \bigg(\frac{\cdot}{\sqrt2} \bigg)
 G \bigg(\frac{\cdot}{\sqrt2}\bigg)\bigg)(y)
\end{aligned}
\end{equation}
for s-a.e. $y\in C_0(Q)$, where   $\mathbf{s}(\mathcal H,k_1)$ 
and $\mathbf{s}(\mathcal H,k_2)$
are functions in $\mathrm{Supp}_{BV}(Q)$ satisfying the relations
\[
\mathbf{s}^2(\mathcal H,k_1)
\equiv \mathbf{s}^2(h_1,\ldots,h_n,k_1)
=h_1^2 +\cdots+h_n^2 +k_1^2
\]
and
\[
\mathbf{s}(\mathcal H,k_2)^2\equiv \mathbf{s}(h_1,\ldots,h_n,k_2)^2
=h_1^2 +\cdots+h_n^2 +k_2^2 
\]
for $m_L^2$-a.e. on $Q$, respectively.
\end{theorem}
 \begin{proof}
Applying  \eqref{eq:gfft-n-fubini}, the first equality of \eqref{eq:multi-rel-01}
follows immediately. Next using \eqref{eq:cp-fft-basic} with $h$ replaced with $\mathbf{s}(\mathcal H)$,
the second equality of \eqref{eq:multi-rel-01} also follows.
 \qed\end{proof}

   \par
 In view of equations \eqref{eq:cp-fft-basic} and \eqref{eq:gfft-n-fubini},
 we also obtain the following assertion.

\renewcommand{\thetheorem}{\thesection.13}

\begin{theorem} \label{thm:iter-gfft-gcp-compose-2nd}
Let $F$ and $G$ be as in Theorem  \ref{thm:gcp}. 
Given a  function  $h$ in $\mathrm{Supp}_{BV}(Q)$
and finite sequences $\mathcal K_1=\{k_{11},k_{12},\ldots,k_{1n}\}$
and  $\mathcal K_2=\{k_{21},k_{22}$, $\ldots,k_{2m}\}$ of    functions in $\mathrm{Supp}_{BV}(Q)$, 
assume that  
\[
h^2=\mathbf{s}(\mathcal K_1)\mathbf{s}(\mathcal K_2)
\] 
for $m_L^2$-a.e. on $Q$. 
Then,  for all   nonzero real   $q$, 
\[
\begin{aligned}
&\Big(T_{q,h/\sqrt2}^{(1)}
\big(T_{q,k_{1n}/\sqrt2}^{(1)}   
\big(\cdots
\big(T_{q,k_{12}/\sqrt2}^{(1)}\big(T_{q,k_{11}/\sqrt2}^{(1)}(F)\big)\big)\cdots\big)\big) \\
&\quad 
*T_{q,h/\sqrt2}^{(1)}\big( T_{q,k_{2m}/\sqrt2}^{(1)}  
\big( \cdots
\big(T_{q,k_{22}/\sqrt2}^{(1)}\big(T_{q,k_{21}/\sqrt2}^{(1)}(G)\big)\big)\cdots\big)\big)\Big)_{-q}^{(\mathbf{s}(\mathcal K_1),\mathbf{s}(\mathcal K_2))}(y)\\
&=\Big(T_{q,h/\sqrt2}^{(1)}\big(T_{q,\mathbf{s}(\mathcal K_1)/\sqrt2}^{(1)}(F)\big) 
*T_{q,h/\sqrt2}^{(1)}\big( T_{q,\mathbf{s}(\mathcal K_2)/\sqrt2}^{(1)}(G)\big) \Big)_{-q}^{(\mathbf{s}(\mathcal K_1),\mathbf{s}(\mathcal K_2))}(y)\\
&=\Big(T_{q, \mathbf{s}(h,\mathbf{s}(\mathcal K_1))/\sqrt2}^{(1)}(F)
*T_{q,\mathbf{s}(h,\mathbf{s}(\mathcal K_2))/\sqrt2}^{(1)}(G)\Big)_{-q}^{(\mathbf{s}(\mathcal K_1),\mathbf{s}(\mathcal K_2))}(y)\\
&= T_{q,h}^{(1)} \bigg(F \bigg(\frac{\cdot}{\sqrt2} \bigg)
 G \bigg(\frac{\cdot}{\sqrt2}\bigg)\bigg)(y)
\end{aligned}
\]
for s-a.e. $y\in C_0(Q)$, where 
$\mathbf{s}(h,\mathbf{s}(\mathcal K_1))$, and $\mathbf{s}(h,\mathbf{s}(\mathcal K_2))$
are functions in $\mathrm{Supp}_{BV}(Q)$ satisfying the relations
\[
\mathbf{s}^2(h,\mathbf{s}(\mathcal K_1))
=h^2 +\mathbf{s}^2(\mathcal K_1)=h^2+k_{11}^2 + \cdots+k_{1n}^2, 
\]
and
\[
\mathbf{s}^2(h,\mathbf{s}(\mathcal K_2))
=h^2 +\mathbf{s}^2(\mathcal K_2)=h^2+k_{21}^2   +\cdots+k_{2m}^2
\]
for $m_L^2$-a.e. on $Q$, respectively.
\end{theorem}

\section*{Acknowledgements}
The author would like to express his gratitude to the editor and the referees
for their valuable comments and suggestions which have improved the original paper.

\end{document}